\documentclass[10pt]{amsart} 
\usepackage{amsmath,amssymb,bm, color}

\usepackage{amsmath}
\usepackage{graphicx,float} 

\usepackage{mathrsfs}
\usepackage{bbm}
\usepackage{bm}

\usepackage{amsmath}
\usepackage{graphicx} 
\usepackage{mathrsfs}
\usepackage{appendix}
\usepackage{float}
\usepackage{amsfonts,amssymb}
\usepackage{dsfont}
\usepackage{pifont}
\usepackage{wrapfig} 
\usepackage{hyperref}
\usepackage{multirow}
\numberwithin{equation}{section}
\def\3bar{{|\hspace{-.02in}|\hspace{-.02in}|}}

\def\W{{\mathcal{W}}}

\def\bmu{{\boldsymbol{\mu}}}

\def\cal{\mathcal}

\newtheorem{lemma}{Lemma}[section]
\newtheorem{theorem}{Theorem}[section]

\newtheorem{algorithm}{Least Squares Weak Galerkin Algorithm}[section]

 \def\ad#1{\begin{aligned}#1\end{aligned}}   
 
\def\an#1{\begin{align}#1\end{align}}

\def\p#1{\begin{pmatrix}#1\end{pmatrix}}

\title[]
 {A Least Squares Weak Galerkin Finite Element Method for Fokker-Planck Type Equations}

  \author {Chunmei Wang}
  \address{Department of Mathematics, University of Florida, Gainesville, FL 32611, USA. }
  \email{chunmei.wang@ufl.edu}

\author {Shangyou Zhang}
\address{Department of Mathematical Sciences,  University of Delaware, Newark, DE 19716, USA}   \email{szhang@udel.edu}

\begin{document}

\begin{abstract}
This paper presents a least squares weak Galerkin (LS-WG) finite element method for a class of second order elliptic equations of Fokker-Planck type. To address the numerical challenges arising from non-smooth diffusion tensors, the proposed method utilizes a least-squares formulation that yields a symmetric positive definite (SPD) discrete system. The numerical scheme is designed by employing locally constructed weak second order partial derivatives and the weak divergence commonly used within the weak Galerkin framework. A rigorous theoretical foundation is provided, establishing the uniqueness of the discrete solution and deriving optimal-order error estimates in a discrete energy norm. Finally, extensive numerical experiments are reported to validate the theoretical findings and demonstrate the robustness and performance of the numerical scheme.
\end{abstract}



\maketitle

\section{Introduction} 
Arising from the study of stochastic processes, the Fokker-Planck equation provides an indispensable mathematical apparatus for analyzing statistical fluctuations across physical systems, engineering applications, and biological complexes \cite{fokker,planck,Risken,stratonovich,gardiner}. In the context of statistical mechanics, this second-order partial differential equation governs the temporal evolution of a probability density function for a particle moving under the coupled effects of deterministic drag forces and stochastic disturbances modeled as Gaussian white noise. Structurally, given an open spatial domain $\Omega\subset{\mathbb R}^d$ ($d$-dimensional Euclidean space) and a terminal time horizon $T$, the objective is to find a time-dependent distribution function $p=p(x,t): \ \Omega\times [0,T]\to {\mathbb R}$ that satisfies
\begin{equation}\label{EQ:FPE}
\begin{split}
\partial_t p +\nabla\cdot (\bmu p)-\frac{1}{2}\sum_{i,j=1}^d
\partial_{ij}^2(a_{ij}p) & = \ 0,\qquad
t\in (0, T), \ x\in \Omega,\\
p(x, 0) & = \ p_0(x),\qquad x\in \Omega,
\end{split}
\end{equation}
where $\partial_{ij}^2=\frac{\partial}{\partial x_j}\frac{\partial}{\partial x_i}$ signifies the second-order partial differentiation with respect to coordinates $x_i$ and $x_j$. Within this system, $a(x)=\{a_{ij}(x)\}_{d\times d}$ denotes the diffusion tensor, $\bmu=(\mu_1,\cdots, \mu_d)$ captures the drift vector field, and $p_0=p_0(x)$ indicates the initial density profile. To ensure well-posedness, \eqref{EQ:FPE} is typically closed using one of two standard boundary conditions: a Dirichlet constraint on the density or a Neumann condition on the total flux. A homogeneous Dirichlet boundary condition represents an absorbing barrier where particles immediately exit the system upon impact, whereas a prescribed Neumann flux dictates a known net current of particles moving across the boundary along the outward normal direction.

Designing accurate numerical schemes for the Fokker-Planck equation involves overcoming several fundamental mathematical and computational challenges. Chief among these issues are the curse of high dimensionality and the low regularity inherent to the target probability density function. For example, in multi-particle statistical mechanics, the Fokker-Planck framework models a joint probability density spanning a vast phase space, making the spatial dimension $d$ prohibitively large. Concurrently, when the diffusion coefficients $a(x)$ are non-smooth, the resulting density profile $p(x,t)$ often develops steep, shock-like discontinuities that are notoriously difficult to resolve on discrete meshes. Beyond capturing these sharp interfaces, any viable numerical algorithm must also strictly respect physical invariants, including total mass conservation and the pointwise non-negativity of the computed density field.

To simulate these probability fields, a variety of finite element techniques have been proposed in the literature; see \cite{bhandari-sherrer,langley,langtangen,bergman,spencer-bergman,masud,kumar-narayana} and the references therein. For instance, in \cite{bhandari-sherrer,langley}, the stationary version of the equation was treated using standard Galerkin finite element methods based on a weak formulation derived via classical integration by parts. To enforce the global integral constraint associated with probability preservation, the authors in \cite{langtangen} augmented a Galerkin scheme with a generalized Lagrange multiplier approach. In \cite{bergman}, conventional continuous $C^0$ elements were applied to Fokker-Planck systems driven by both additive and multiplicative white noise processes, while a multiscale finite element architecture was introduced in \cite{masud} to analyze stochastic structural dynamics in high dimensions. A major limitation of these existing methodologies is their strict reliance on a highly smooth or constant diffusion tensor $a(x)$, which is necessary to ensure the validity of the underlying weak forms. While a $C^0$ continuous finite element framework was implemented and tested for the transient equation under non-smooth diffusion coefficients in \cite{kumar-narayana}, their work was purely empirical and lacked a formal mathematical convergence analysis.

When the diffusion coefficients possess sufficient regularity, the second-order differential operator can be expanded using the product rule to yield the following equivalent identity:
\begin{equation}\label{EQ:Reformulation}
\frac12 \partial_{ij}^2(a_{ij}p) = \frac12 \partial_j (a_{ij}(x) \partial_i p) + \frac12
\partial_j ( (\partial_i a_{ij}) p).
\end{equation}
This algebraic restructuring effectively transforms \ \eqref{EQ:FPE} into a standard \ time-dependent convection-diffusion equation in divergence form. Setting aside high-dimensional bottlenecks, this reformulated system represents a relatively straightforward problem for conventional numerical methods. However, when the diffusion tensor is non-smooth, the expansion \eqref{EQ:Reformulation} breaks down completely due to the non-differentiability of the coefficients. In this rough regime, the exact solution $p(x,t)$ exhibits unmapped internal discontinuities, making classical finite element approximations highly unstable or inapplicable. The primary objective of this paper is to construct a novel finite element framework that successfully circumvents this structural limitation, providing a stable, accurate, and rigorous approach specifically designed to handle non-smooth diffusion tensors in the Fokker-Planck equation.

To streamline the presentation, we consider a Fokker-Planck type model equation subject to a homogeneous Dirichlet boundary condition. The model problem seeks an unknown function $u=u(x)$ satisfying
\begin{equation}\label{model}
\begin{split}
\nabla \cdot (\bmu u)-\frac{1}{2}\sum_{i,j=1}^d \partial^2_{ij}(a_{ij}u)=&f,\quad \text{in}\ \Omega,\\
u =& 0,\quad \text{on}\ \partial\Omega,
\end{split}
\end{equation}
where $\Omega$ is an open, bounded domain in $\mathbb{R}^d (d=2, 3)$ with a Lipschitz continuous boundary $\partial\Omega$, and $f\in L^2(\Omega)$ is a prescribed source term. 
We assume 1) the diffusion tensor $a(x)=\{a_{ij}(x)\}_{d\times d}\in (W_2^\infty (\Omega))^{d\times d}$ is   symmetric, positive definite, and uniformly bounded across the domain $\Omega$; 2) the drift vector satisfies the regularity condition $\bmu \in (W_1^\infty (\Omega))^d$.

  The Weak Galerkin (WG) finite element method \cite{wg1, wg2, wg3, wg4, wg5, wg6, wg7, wg8, wg9, wg10, wg11, wg12, wg13, wg14, wg15, wg16, wg17, wg18, wg19, wg20, wg21, itera, wz2023, wy3655, guan, guan2} provides a highly adaptable alternative to conventional continuous finite element frameworks. The core mechanism of the WG approach relies on the deployment of discrete weak derivatives, coupled with targeted stabilization forms that enforce weak continuity along element interfaces. This structural flexibility allows the WG methodology to seamlessly operate on arbitrary polygonal and polyhedral mesh configurations. 

Building upon these foundational principles, the Primal-Dual Weak Galerkin (PDWG) method \cite{pdwg1, pdwg2, pdwg3, pdwg4, pdwg5, pdwg6, pdwg7, pdwg8, pdwg9, pdwg10, pdwg11, pdwg12, pdwg13, pdwg14, pdwg15} was designed to cast numerical approximations as constrained optimization problems. The PDWG architecture delivers enhanced numerical stability, particularly when applied to non-self-adjoint systems such as linear transport models \cite{wwhyperbolic}. However, this theoretical advantage comes with a computational trade-off: the necessary inclusion of dual variables inherently expands the total number of global unknowns within the discrete system.

To address this limitation, the proposed Least-Squares Weak Galerkin (LS-WG) method offers several transformative advantages for the Fokker-Planck model problem \eqref{model}. The least-squares formulation naturally yields a symmetric positive-definite (SPD) discrete system. This property is a significant computational asset, permitting the direct application of high-performance iterative solvers such as the Conjugate Gradient (CG) method. By marrying the inherent stability of least-squares minimization with the geometric versatility of weak derivatives, the LS-WG approach provides a robust and scalable numerical platform.

In this work, we provide a rigorous theoretical foundation for the LS-WG scheme, establish the well-posedness and uniqueness of the discrete solution, and derive optimal-order error estimates in a discrete energy norm. Extensive numerical experiments are subsequently presented to validate these theoretical findings and demonstrate the overall robustness of the method.

We will follow the usual notation for Sobolev spaces
and norms. For any open bounded domain $D\subset \mathbb{R}^d$ with
Lipschitz continuous boundary, we use $\|\cdot\|_{s,D}$ and
$|\cdot|_{s,D}$ to denote the norm and seminorms in the Sobolev
space $H^s(D)$ for any $s\ge 0$, respectively. The inner product
in $H^s(D)$ is denoted by $(\cdot,\cdot)_{s,D}$. The space
$H^0(D)$ coincides with $L^2(D)$, for which the norm and the inner
product are denoted by $\|\cdot \|_{D}$ and $(\cdot,\cdot)_{D}$,
respectively. When $D=\Omega$, we shall drop the subscript $D$ in
the norm and inner product notation. 

 The paper is organized as follows. In Section 2, we shall discuss the computation of weak divergence and weak second order partial derivatives. In Section 3, we will present a detailed description of the least squares weak Galerkin finite element method for the Fokker-Planck type model problem   and study the solution existence and uniqueness for our numerical method. In Section 4, we will establish an error estimate in a discrete norm. Finally, in Section 5, we report some numerical results to demonstrate the performance of the LS-WG finite element method.

\section{Weak Differential Operators} 
The goal of this section is to formally introduce the continuous and discrete weak differential operators. To this end, let $T$ be a polygonal or polyhedral element with boundary $\partial T$. By a weak function on $T$, we mean a triplet $v=\{v_0,v_b,\mathbf{v}_g\}$ in which $v_0\in L^2(T)$ represents the   value in the interior of $T$, $v_b\in L^{2}(\partial T)$ represents the value on the boundary of $T$, and $\mathbf{v}_g=(v_{g1},\cdots,v_{gd})\in [L^{2}(\partial T)]^d$ represents the gradient of $v$ on $\partial T$. In general, the boundary components $v_b$ and $\mathbf{v}_g$ are not required to be  equal to the traditional traces of $v_0$ and $\nabla v_0$ on $\partial T$. We denote the space of all well-defined weak functions on $T$ by $\W(T)$:
\begin{equation}\label{2.1}
\W(T)=\{v=\{v_0,v_b,\mathbf{v}_g\}: v_0\in L^2(T), v_b\in L^{2}(\partial T), \mathbf{v}_g\in [L^{2}(\partial T)]^d\}.
\end{equation}

For any $v\in \W(T)$, the weak second-order partial derivative  of $ a_{ij}v$, denoted by $\partial^2_{ij,w} (a_{ij}v)$, is defined as a linear functional in the dual space of $H^2(T)$ satisfying
\begin{equation}\label{2.3}
(\partial^2_{ij,w}(a_{ij}v),\varphi)_T=(a_{ij}v_0,\partial^2_{ji}\varphi)_T- \langle a_{ij}v_b n_i,\partial_j\varphi\rangle_{\partial T}+ \langle a_{ij} v_{gi},\varphi n_j\rangle_{\partial T},
\end{equation}
for all $\varphi\in H^2(T)$. Here, $\mathbf{n}=(n_1,\cdots,n_d)$ denotes the unit outward normal vector to $\partial T$, $(\cdot,\cdot)_T$ represents the standard $L^2$ inner product in $L^2(T)$, and $\langle \cdot, \cdot\rangle_{\partial T}$ is the $L^2$ inner product on the boundary $\partial T$.

Analogously, for any $v\in \W(T)$, the weak divergence of   $\bmu v$, denoted by $\nabla_w \cdot(\bmu v)$, is defined as a linear functional in the dual space of $H^1(T)$ satisfying
\begin{equation}\label{2.4_continuous}
(\nabla_w \cdot(\bmu v), \psi)_T = -(\bmu v_0, \nabla \psi)_T + \langle (\bmu \cdot \mathbf{n}) v_b, \psi \rangle_{\partial T},
\end{equation}
for all $\psi\in H^1(T)$.

Let $P_r(T)$ denote the space of polynomials on $T$ of degree no greater than $r$. A discrete version of the weak second-order partial derivative $\partial^2_{ij,w} (a_{ij}v)$ for $v\in \W(T)$, denoted by $\partial^2_{ij,w,r,T} (a_{ij}v)$, is defined as the unique polynomial in $P_r(T)$ satisfying
\begin{equation}\label{2.4}
(\partial^2_{ij,w,r,T}(a_{ij}v),\varphi)_T=(a_{ij}v_0,\partial^2_{ji}\varphi)_T-\langle a_{ij}v_b n_i,\partial_j\varphi\rangle_{\partial T} +\langle a_{ij}v_{gi},\varphi n_j\rangle_{\partial T},
\end{equation}
for all $\varphi \in P_r(T)$. Employing standard integration by parts, this definition yields the equivalent representation:
\begin{equation}\label{dissec}
\begin{split}
 (\partial^2_{ij,w,r,T}(a_{ij}v),\varphi)_T&=(\partial^2
 _{ ij}(a_{ij}v_0),\varphi)_T+\langle a_{ij} (v_0-v_b) n_i,\partial_j\varphi\rangle_{\partial T}
\\&-\langle  a_{ij}(\partial_i   v_0-v_{gi}),\varphi n_j\rangle_{\partial T},
 \end{split}
 \end{equation} 
for all $\varphi \in P_r(T)$.

Correspondingly, the discrete weak divergence operator $\nabla_{w,r,T} \cdot (\bmu v)$ for any $v\in \W(T)$ is defined as the unique polynomial in $P_r(T)$ that satisfies
\begin{equation}\label{disgradient}
(\nabla_{w,r,T} \cdot (\bmu v), \psi)_T = -(\bmu v_0, \nabla \psi)_T + \langle (\bmu \cdot \mathbf{n}) v_b, \psi \rangle_{\partial T}, \quad \forall \psi \in P_r(T).
\end{equation}
Applying integration by parts to the volume term in \eqref{disgradient} provides the alternative formulation:
\begin{equation}\label{disdiv}
(\nabla_{w,r,T} \cdot (\bmu v), \psi)_T = (\nabla \cdot (\bmu v_0), \psi)_T + \langle (\bmu \cdot \mathbf{n})(v_b - v_0), \psi \rangle_{\partial T}, \quad \forall \psi \in P_r(T).
\end{equation}

\section{Least Squares Weak Galerkin   Scheme}\label{Section:WGFEM}
The primary objective of this section is to formulate a least-squares weak Galerkin  finite element method for the model problem \eqref{model}. To this end, let $\mathcal{T}_h$ be a finite element partition of the domain $\Omega$ consisting of polygons in $\mathbb{R}^2$ or polyhedra in $\mathbb{R}^3$, which is shape-regular in the sense described in \cite{wy3655}. We denote the set of all edges (in 2D) or faces (in 3D) in $\mathcal{T}_h$ by $\mathcal{E}_h$, and define the set of all interior edges or faces as $\mathcal{E}_h^0=\mathcal{E}_h\setminus\partial\Omega$. For each element $T\in\mathcal{T}_h$, $h_T$ denotes its diameter, and the global mesh size of the partition is given by $h=\max_{T\in\mathcal{T}_h}h_T$.

Let $k\ge 1$ be a specified integer. On each element $T\in\mathcal{T}_h$, we define the local discrete weak function space $V(T)$ as:
\an{\label{V-T}
  V(T)= \{\{v_0, v_b, \mathbf{v}_g\}:\ v_0\in P_k(T),\ v_b\in P_k(e),\ \mathbf{v}_g\in [P_{k-1}(e)]^d,\ e\subset \partial T \}. }
By patching $V(T)$ over all elements $T\in\mathcal{T}_h$ via common value of $v_b$   on the interior edges or faces, we construct the global weak finite element space $V_h$:
$$V_h= \{\{v_0, v_b, \mathbf{v}_g\}:\ \{v_0, v_b, \mathbf{v}_g\}|_T\in V(T),\ \forall T\in\mathcal{T}_h \}.$$
To incorporate the homogeneous Dirichlet boundary condition, we define $V_h^0$ as the subspace of $V_h$ with vanishing boundary value for $v_b$ on $\partial\Omega$:
$$V_h^0= \{\{v_0, v_b, \mathbf{v}_g\}\in V_h:\ v_b|_e=0,\ \forall e\subset \partial\Omega \}.$$

For simplicity of notation, we denote the discrete weak divergence $\nabla_{w,k-1,T}\cdot(\bmu\sigma)$ computed using \eqref{disgradient} on each element $T$ with $r=k-1$ simply by $\nabla_w\cdot(\bmu\sigma)$:
$$(\nabla_w\cdot(\bmu\sigma))|_T=\nabla_{w,k-1,T}\cdot(\bmu\sigma|_T),\qquad \forall \sigma\in V_h.$$
Analogously, $\partial^2_{ij, w}(a_{ij}\sigma)$ denotes the discrete weak second-order partial derivative $\partial^2_{ij,w,k-1,T}(a_{ij}\sigma)$ computed via \eqref{2.4} on each element $T$ with $r=s$ where  $s=k-1$ or $s=k-2$ for $k\geq 2$ and $s=k-1$ for $k=1$:
$$(\partial^2_{ij, w}(a_{ij}\sigma))|_T=\partial^2_{ij,w,s,T}(a_{ij}\sigma|_T),\qquad \forall \sigma\in V_h.$$

Next, we introduce the following global bilinear forms on $V_h\times V_h$:
$$\begin{aligned}
a(u, v) &= \sum_{T\in\mathcal{T}_h} a_T(u, v), \\
s(u, v) &= \sum_{T\in\mathcal{T}_h} s_T(u, v),
\end{aligned}$$
where the local bilinear forms are defined on each element $T$ as
$$a_T(u, v)= (\nabla_w\cdot(\bmu u)-\frac{1}{2}\sum_{i,j=1}^d \partial^2_{ij, w}(a_{ij}u),\ \nabla_w\cdot(\bmu v)-\frac{1}{2}\sum_{i,j=1}^d \partial^2_{ij, w}(a_{ij}v) )_T,$$
and
$$s_T(u, v)=h_T^{-3}\langle u_0-u_b, v_0-v_b\rangle_{\partial T} + h_T^{-1}\langle\nabla u_0-\mathbf{u}_g, \nabla v_0-\mathbf{v}_g\rangle_{\partial T}.$$

We are now in a position to state the least-squares weak Galerkin finite element scheme for the model problem \eqref{model}.

\begin{algorithm}\label{lswg}
Let $k\ge 1$ be a given integer. The numerical approximation for the solution of \eqref{model} seeks $u_h\in V_h^0$ satisfying
\begin{equation}\label{ls}
a(u_h, v) + s(u_h, v) = \sum_{T\in\mathcal{T}_h} (f,\ \nabla_w\cdot(\bmu v)-\frac{1}{2}\sum_{i,j=1}^d \partial^2_{ij, w}(a_{ij}v) )_T,
\end{equation}
for all $v\in V_h^0$.
\end{algorithm}

The well-posedness of this numerical scheme is established by the following lemma.

\begin{lemma}\label{norm1}
Assume the continuous model problem \eqref{model} possesses a unique solution. Then, the least-squares weak Galerkin scheme defined in Algorithm \ref{lswg} has a unique solution $u_h\in V_h^0$.
\end{lemma}

\begin{proof}
Since the discrete system is square and finite-dimensional, it suffices to show that the solution to \eqref{ls} is identically zero when $f=0$. Assuming $f=0$ and choosing the test function $v=u_h$ in \eqref{ls}, we obtain
$$\sum_{T\in\mathcal{T}_h} \|\nabla_w\cdot(\bmu u_h)-\frac{1}{2}\sum_{i,j=1}^d \partial^2_{ij, w}(a_{ij}u_h) \|_T^2 + s(u_h, u_h)=0.$$
Because both terms on the left-hand side are strictly non-negative, this implies that on each element $T\in\mathcal{T}_h$,
\begin{equation}\label{eq:proof1}
\nabla_w\cdot(\bmu u_h)-\frac{1}{2}\sum_{i,j=1}^d \partial^2_{ij, w}(a_{ij}u_h)=0,
\end{equation}
and from the stabilization term $s(u_h, u_h)=0$, we deduce
\begin{equation}\label{eq:proof2}
u_0=u_b \quad \text{and} \quad \nabla u_0=\mathbf{u}_g \quad \text{on}\ \partial T.
\end{equation}
Combining \eqref{eq:proof2} with the  definitions of the discrete weak operators in \eqref{dissec} and \eqref{disdiv} yields
$$\nabla_w\cdot(\bmu u_h)=\nabla\cdot(\bmu u_0), \qquad \partial^2_{ij, w}(a_{ij}u_h)=\partial^2_{ij}(a_{ij}u_0) \quad \text{in}\ T$$
Applying these identities directly to \eqref{eq:proof1}, we see that $u_0$ satisfies the governing partial differential equation strongly within each element:
\begin{equation}\label{eq:proof3}
\nabla\cdot(\bmu u_0)-\frac{1}{2}\sum_{i,j=1}^d \partial^2_{ij}(a_{ij}u_0)=0 \quad \text{in}\ T.
\end{equation}
Moreover, the constraint equations $u_0=u_b$ and $\nabla u_0=\mathbf{u}_g$ on each $\partial T$ guarantee that $u_0$ and its gradient $\nabla u_0$ continuously match the single-valued boundary functions $u_b$ and $\mathbf{u}_g$ across all interior interfaces. This global $C^1$-continuity implies that $u_0\in H^2(\Omega)$. Therefore, the local relationship established in \eqref{eq:proof3} holds globally across $\Omega$.

Additionally, $u_0=u_b$ on each $\partial T$, and since $u_h\in V_h^0$, we are given $u_b=0$ on $\partial\Omega$. This enforces the boundary condition $u_0=0$ on $\partial\Omega$. By the uniqueness assumption of the continuous model problem, it follows that $u_0\equiv 0$ everywhere in $\Omega$. This subsequently forces $u_b\equiv 0$ and $\mathbf{u}_g\equiv\mathbf{0}$ globally, which completes the proof.
\end{proof}

Finally, we introduce a semi-norm on the space $V_h$ defined by:
$$\3bar v\3bar^2=a(v, v) + s(v, v).$$
Following the exact same algebraic arguments utilized in the proof of Lemma \ref{norm1}, one can easily verify that $\3bar\cdot\3bar$ defines a rigorous norm on the constrained subspace $V_h^0$.

\section{Error Estimates}
In this section, we establish optimal-order error estimates for the proposed LS-WG finite element method. Throughout the analysis, the mesh partition $\mathcal{T}_h$ is assumed to be shape-regular. For any element $T \in \mathcal{T}_h$ and any function $\phi \in H^1(T)$, we rely on the following standard continuous trace inequality \cite{wy3655}:
\begin{equation}\label{tracein}
\|\phi\|^2_{\partial T} \leq C  ( h_T^{-1}\|\phi\|_T^2 + h_T \|\nabla \phi\|_T^2  ).
\end{equation}

\begin{lemma}
For any $u, v \in V_h$, the bilinear forms satisfy the following Cauchy-Schwarz type boundedness property:
\begin{equation*}
a(u, v) + s(u, v) \leq \3bar u \3bar \, \3bar v \3bar.
\end{equation*}
\end{lemma}
\begin{proof}
    This proof can be achieved easily by using Cauchy-Schwarz inequality. 
\end{proof}

On each element $T \in \mathcal{T}_h$, let $Q_0$ denote the standard $L^2$ projection onto the polynomial space $P_k(T)$ for $k \geq 1$. Similarly, on each edge or face $e \subset \partial T$, let $Q_b$ and $\mathbf{Q}_g := (Q_{g1}, \cdots, Q_{gd})$ denote the $L^2$ projections onto $P_k(e)$ and $[P_{k-1}(e)]^d$, respectively. For any sufficiently smooth function $w \in H^2(\Omega)$, we define the global  projection operator $Q_h w \in V_h$ such that its restriction to each element $T$ is given by:
\begin{equation*} 
Q_h w = \{Q_0 w, Q_b w, \mathbf{Q}_g(\nabla w)\}.
\end{equation*}

For simplicity, we assume  that the diffusion tensor $a=\{a_{ij}\}_{d\times d}$ and the convection coefficient $\bmu$ are piecewise constant. The analysis readily extends to the case where $a$ and $\bmu$ are non-smooth coefficients.

\begin{lemma}\label{Lemma5.1}
 Let $\mathcal{Q}^{(k-1)}_h$ and $\mathcal{Q}^{(s)}_h$  be the $L^2$ projections onto $P_{k-1}(T)$ and $P_{s}(T)$ respectively.   The defined projection operators satisfy the following commutative properties on each element $T \in \mathcal{T}_h$:
\begin{align}
\partial^2_{ij,w}(a_{ij}Q_h w) &= \mathcal{Q}^{(s)}_h(\partial^2_{ij} (a_{ij} w)), \qquad i,j = 1, \ldots, d, \label{l} \\
\nabla_w \cdot (\bmu Q_h w) &= \mathcal{Q}^{(k-1)}_h(\nabla \cdot (\bmu w)). \label{l-2}
\end{align}
\end{lemma}

\begin{proof}
For any test function $\varphi \in P_{s}(T)$ and for any sufficiently smooth $w$, we utilize the definition of the weak second order partial derivative \eqref{2.4}, the properties of the $L^2$ projections $Q_0$, $Q_b$, and $\mathbf{Q}_g$, and standard integration by parts to obtain:
\begin{equation*}
\begin{aligned} &\quad \ 
(\partial^2_{ij,w}(a_{ij}Q_h w), \varphi)_T \\
&= (a_{ij}Q_0 w, \partial^2_{ji}\varphi)_T - \langle a_{ij}Q_b w, \partial_j \varphi \, n_i \rangle_{\partial T} + \langle a_{ij} Q_{gi}(\partial_i w) \, n_j, \varphi \rangle_{\partial T} \\
&= (a_{ij}w, \partial^2_{ji}\varphi)_T - \langle a_{ij} w, \partial_j \varphi \, n_i \rangle_{\partial T} + \langle a_{ij} \partial_i w \, n_j, \varphi \rangle_{\partial T} \\
&= (\partial^2_{ij}(a_{ij}w), \varphi)_T \\
&= (\mathcal{Q}^{(s)}_h(  \partial^2_{ij}(a_{ij}w)), \varphi)_T,
\end{aligned}
\end{equation*}
which directly confirms \eqref{l}.

Similarly, for any test function $\psi \in P_{k-1}(T)$ and smooth $w$, employing the discrete weak divergence definition \eqref{disgradient}, the properties of the $L^2$ projections $Q_0$, $Q_b$, and the integration by parts formula yields:
\begin{equation*}
\begin{aligned}
(\nabla_w \cdot (\bmu Q_h w), \psi)_T &= -(\bmu Q_0w, \nabla \psi)_T + \langle (\bmu \cdot \mathbf{n}) Q_bw, \psi \rangle_{\partial T} \\
&= -(\bmu w, \nabla \psi)_T + \langle (\bmu \cdot \mathbf{n}) w, \psi \rangle_{\partial T} \\
&= (\nabla \cdot (\bmu w), \psi)_T \\
&= (\mathcal{Q}^{(k-1)}_h(\nabla \cdot (\bmu w)), \psi)_T.
\end{aligned}
\end{equation*}
This establishes \eqref{l-2} and completes the proof.
\end{proof}

\begin{theorem}
Let $u_h \in V_h^0$ be the weak Galerkin finite element solution to the model problem \eqref{model} generated by the LS-WG Algorithm \ref{lswg}, and let $u \in H^{k+1}(\Omega)$ be the exact solution of the model problem \eqref{model}. Then, there exists a constant $C > 0$, independent of $h$, such that
\begin{equation}\label{estimate1}
\3bar Q_h u - u_h \3bar \leq C h^k \|u\|_{k+1}.
\end{equation}
\end{theorem}

\begin{proof}
Testing the strong form of the governing equation in \eqref{model} against the discrete weak differential operator $\nabla_w \cdot (\bmu v) - \frac{1}{2}\sum_{i,j=1}^d \partial^2_{ij,w}(a_{ij}v)$ for any $v \in V_h^0$ gives:
\begin{equation*}
\begin{split}
     &
    \sum_{T\in {\cal T}_h}(\nabla \cdot (\bmu u) - \frac{1}{2}\sum_{i,j=1}^d \partial^2_{ij}(a_{ij}u), \, \nabla_w \cdot (\bmu v) - \frac{1}{2}\sum_{i,j=1}^d \partial^2_{ij,w}(a_{ij}v) )_T\\
 =&   \sum_{T\in {\cal T}_h} (f, \nabla_w \cdot (\bmu v) - \frac{1}{2}\sum_{i,j=1}^d \partial^2_{ij,w}(a_{ij}v) )_T.
\end{split}
\end{equation*}
Because   $\nabla_w \cdot (\bmu v) - \frac{1}{2}\sum_{i,j=1}^d \partial^2_{ij,w}(a_{ij}v) \in P_{k-1}(T)$ on each element $T$, we can directly apply the commuting properties \eqref{l} and \eqref{l-2} to rewrite the left-hand side:
\begin{equation*}
a(Q_h u, v) =    \sum_{T\in {\cal T}_h}(f, \nabla_w \cdot (\bmu v) - \frac{1}{2}\sum_{i,j=1}^d \partial^2_{ij,w}(a_{ij}v) )_T.
\end{equation*}
Adding the stabilization term $s(Q_h u, v)$ to both sides of the equation yields:
\begin{equation}\label{eq:error_step1}
a(Q_h u, v) + s(Q_h u, v) = \sum_{T \in \mathcal{T}_h}  (f, \nabla_w \cdot (\bmu v) - \frac{1}{2}\sum_{i,j=1}^d \partial^2_{ij,w}(a_{ij}v) )_T + s(Q_h u, v).
\end{equation}
Subtracting the discrete LS-WG numerical scheme \eqref{ls} from \eqref{eq:error_step1}, we obtain the error equation:
\begin{equation*}
a(e_h, v) + s(e_h, v) = s(Q_h u, v),
\end{equation*}
where $e_h = Q_h u - u_h$. Taking the test function $v = e_h$ yields the relation:
\begin{equation}\label{es1}
\3bar e_h \3bar^2 = s(Q_h u, e_h).
\end{equation}

Applying the Cauchy-Schwarz inequality and the  trace inequality \eqref{tracein}, we   deduce:
\begin{equation}\label{sa}
\begin{aligned}
&\quad s(Q_h u, e_h)\\ &\leq s^{1/2}(Q_h u, Q_h u) s^{1/2}(e_h, e_h) \\
&\leq  ( \sum_{T\in \mathcal{T}_h} h_T^{-3}\|Q_0u -Q_b u\|^2_{\partial T} + h_T^{-1} \|\nabla Q_0u - \mathbf{Q}_g(\nabla u)\|^2_{\partial T}  )^{1/2} \3bar e_h \3bar \\
&\leq C  ( \sum_{T\in \mathcal{T}_h} h_T^{-4}\|Q_0u - u\|^2_T +h_T^{-2} \|Q_0u - u\|^2_{1, T} + h_T^{-2}\|\nabla Q_0u - \nabla u\|^2_T\\& +  \|\nabla Q_0u - \nabla u\|^2_{1,T}  )^{1/2} \3bar e_h \3bar \\
&\leq C h^{k-1} \|u\|_{k+1} \3bar e_h \3bar.
\end{aligned}
\end{equation}
Dividing both sides by $\3bar e_h \3bar$ and combining this result with \eqref{es1} completes the proof.
\end{proof}

\begin{theorem}
Let $u \in H^{k+1}(\Omega)$ be the exact solution of \eqref{model}, and let $u_h \in V_h^0$ be its numerical approximation obtained from \eqref{ls}. There exists a constant $C > 0$, independent of $h$, such that
\begin{equation}
\3bar u - u_h \3bar \leq C h^{k-1} \|u\|_{k+1}.
\end{equation}
\end{theorem}

\begin{proof}
Applying the triangle inequality to the discrete energy norm, we have:
\begin{equation}\label{eq:triangle_error}
\3bar u - u_h \3bar \leq \3bar u - Q_h u \3bar + \3bar Q_h u - u_h \3bar.
\end{equation}
Using the definitions of the  bilinear forms and \eqref{sa}, the approximation error can be expanded and bounded via standard projection estimates:
\begin{equation*}
\begin{aligned}
\3bar u - Q_h u \3bar &\leq  ( \sum_{T\in \mathcal{T}_h} \|\nabla \cdot (\bmu u) - \mathcal{Q}_h^{(k-1)} \nabla \cdot (\bmu u)\|_T^2\\
&+  \|\frac{1}{2} \sum_{i,j=1}^d (\partial_{ij}^2 (a_{ij}u) - \mathcal{Q}_h^{(s)} \partial_{ij}^2 (a_{ij}u)) \|_T^2 + s(Q_h u, Q_h u)  )^{1/2} \\
&\leq C h^k \|u\|_{k+1} + C h^{k-1} \|u\|_{k+1} \leq
C h^{k-1} \|u\|_{k+1}.
\end{aligned}
\end{equation*}
Substituting this bound along with the   error estimate \eqref{estimate1} into \eqref{eq:triangle_error} yields:
\begin{equation*}
\begin{aligned}
\3bar u - u_h \3bar &\leq C h^{k-1} \|u\|_{k+1}.
\end{aligned}
\end{equation*}
This concludes the proof.
\end{proof}

\section{Numerical experiment}

In the numerical test,  we solve the model Fokker-Planck problem \eqref{model}, where 
   $\Omega=(0,1)\times(0,1)$, and
\an{\label{sol1} \ad{ \ \bmu&=\p{0\\0}, \ \ (a_{ij})=\p{4&0 \\ 0& 4}, 
          \ \ u = 2^8(x-x^2)^2(y-y^2)^2,  }  }
            or
\an{\label{sol2} \ad{ \ \bmu&=\p{1\\1}, \ \ (a_{ij})=\p{4&2 \\ 2& 4}, 
          \ \ u = 2^8(x-x^2)^2(y-y^2)^2.  }  }

The computations are done on four types of grids, ordered by their level of difficulty: the square grids shown in Figure \ref{f-g1}, the triangular grids shown in Figure \ref{f-g2}, the quadrilateral grids shown in Figure \ref{f-g3}, and the non-convex pentagonal grids shown in Figure \ref{f-g4}. We apply the least squares weak Galerkin method \eqref{ls} using \(P_{k}\) polynomial spaces (\(k=2, 3, 4\)) in \eqref{V-T} to solve the two problems. The results are listed in Tables \ref{t1}-\ref{t8}, which show that the optimal orders of convergence are achieved. It appears that the solutions for \eqref{sol1} are consistently more accurate than those for \eqref{sol2}.

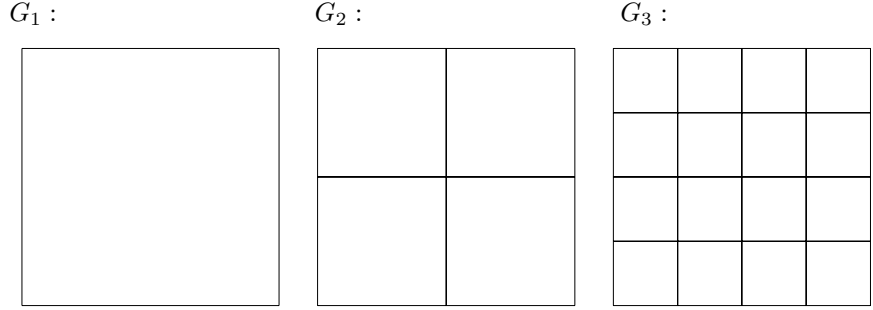
\begin{figure}[H]
 \begin{center}\setlength\unitlength{1.0pt}
\begin{picture}(340,110)(0,0) 
 \put(0,102){$G_1:$}  \put(115,102){$G_2:$} \put(230,102){$G_3:$} 
  
  \put(0,-10){\includegraphics[width=330pt]{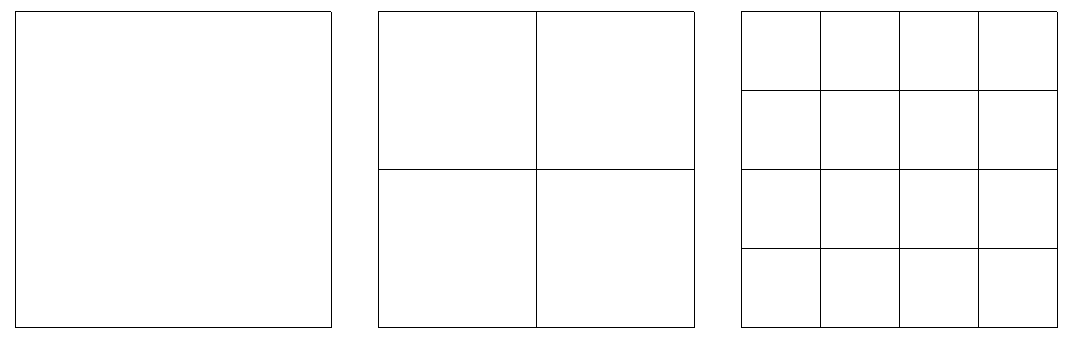}}   
 \end{picture}\end{center}
\caption{The square grids, used in Tables \ref{t1}--\ref{t2}. }\label{f-g1}
\end{figure}

\begin{table}[H]
  \centering  \renewcommand{\arraystretch}{1.1}
  \caption{Error profile for computing \eqref{sol1}, on square grids in Figure \ref{f-g1}. }
  \label{t1}
\begin{tabular}{c|cc|cc}
\hline
  $G_i$ & \quad $\| u-u_h\|_{0}$ & $O(h^r)$ & \  $\3bar u-u_h\3bar $& $O(h^r)$  \\ \hline
    &  \multicolumn{4}{c}{By the $P_2$ WG finite element, $k=2$ in \eqref{V-T}  }   \\
\hline  
 1&    0.146E+01 & --- &    0.754E+02 & --- \\
 2&    0.220E+00 &  2.7&    0.260E+02 &  1.5\\
 3&    0.329E-01 &  2.7&    0.167E+02 &  0.6\\
 4&    0.426E-02 &  3.0&    0.901E+01 &  0.9\\
 5&    0.533E-03 &  3.0&    0.461E+01 &  1.0\\
 6&    0.660E-04 &  3.0&    0.232E+01 &  1.0\\
\hline 
    &  \multicolumn{4}{c}{By the $P_3$ WG finite element, $k=3$ in \eqref{V-T}  }   \\
\hline  
 1&    0.158E+01 & --- &    0.845E+02 & --- \\
 2&    0.178E+00 &  3.1&    0.259E+02 &  1.7\\
 3&    0.142E-01 &  3.6&    0.694E+01 &  1.9\\
 4&    0.947E-03 &  3.9&    0.184E+01 &  1.9\\
 5&    0.594E-04 &  4.0&    0.467E+00 &  2.0\\
 \hline  
    &  \multicolumn{4}{c}{By the $P_4$ WG finite element, $k=4$ in \eqref{V-T}  }   \\
\hline  
 1&    0.246E+01 & --- &    0.723E+03 & --- \\
 2&    0.132E+00 &  4.2&    0.141E+03 &  2.4\\
 3&    0.510E-02 &  4.7&    0.243E+02 &  2.5\\
 4&    0.171E-03 &  4.9&    0.324E+01 &  2.9\\
 5&    0.542E-05 &  5.0&    0.412E+00 &  3.0\\
\hline 
    \end{tabular}%
\end{table}%

\begin{table}[H]
  \centering  \renewcommand{\arraystretch}{1.1}
  \caption{Error profile for computing \eqref{sol2}, on square grids in Figure \ref{f-g1}. }
  \label{t2}
\begin{tabular}{c|cc|cc}
\hline
  $G_i$ & \quad $\| u-u_h\|_{0}$ & $O(h^r)$ & \  $\3bar u-u_h\3bar $& $O(h^r)$  \\ \hline
    &  \multicolumn{4}{c}{By the $P_2$ WG finite element, $k=2$ in \eqref{V-T}  }   \\
\hline  
 1&    0.667E+01 & --- &    0.196E+03 & --- \\
 2&    0.319E+00 &  4.4&    0.581E+02 &  1.8\\
 3&    0.334E-01 &  3.3&    0.357E+02 &  0.7\\
 4&    0.435E-02 &  2.9&    0.186E+02 &  0.9\\
 5&    0.535E-03 &  3.0&    0.940E+01 &  1.0\\
 6&    0.658E-04 &  3.0&    0.471E+01 &  1.0\\
\hline 
    &  \multicolumn{4}{c}{By the $P_3$ WG finite element, $k=3$ in \eqref{V-T}  }   \\
\hline  
 1&    0.474E+01 & --- &    0.396E+03 & --- \\
 2&    0.193E+00 &  4.6&    0.167E+03 &  1.2\\
 3&    0.139E-01 &  3.8&    0.555E+02 &  1.6\\
 4&    0.951E-03 &  3.9&    0.152E+02 &  1.9\\
 5&    0.607E-04 &  4.0&    0.388E+01 &  2.0\\
 \hline  
    &  \multicolumn{4}{c}{By the $P_4$ WG finite element, $k=4$ in \eqref{V-T}  }   \\
\hline  
 1&    0.332E+01 & --- &    0.160E+04 & --- \\
 2&    0.161E+00 &  4.4&    0.301E+03 &  2.4\\
 3&    0.520E-02 &  5.0&    0.514E+02 &  2.6\\
 4&    0.171E-03 &  4.9&    0.684E+01 &  2.9\\
 5&    0.564E-05 &  4.9&    0.868E+00 &  3.0\\
\hline 
    \end{tabular}%
\end{table}%

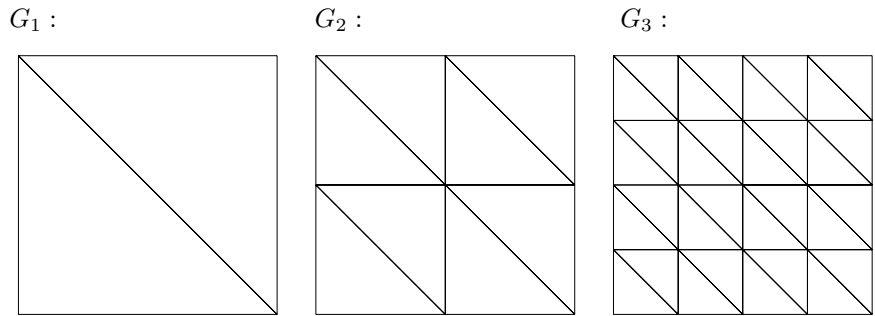
\begin{figure}[H]
 \begin{center}\setlength\unitlength{1.0pt}
\begin{picture}(340,110)(0,0) 
 \put(0,102){$G_1:$}  \put(115,102){$G_2:$} \put(230,102){$G_3:$} 
  
  \put(0,-10){\includegraphics[width=330pt]{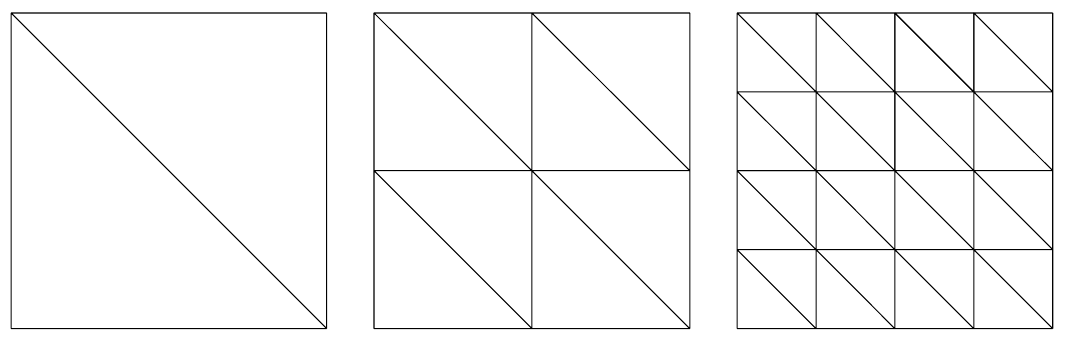}}   
 \end{picture}\end{center}
\caption{The triangular grids, used in Tables \ref{t3}--\ref{t4}. }\label{f-g2}
\end{figure}

\begin{table}[H]
  \centering  \renewcommand{\arraystretch}{1.02}
  \caption{Error profile for \eqref{sol1}, on triangular grids in Figure \ref{f-g2}. }
  \label{t3}
\begin{tabular}{c|cc|cc}
\hline
  $G_i$ & \quad $\| u-u_h\|_{0}$ & $O(h^r)$ & \  $\3bar u-u_h\3bar $& $O(h^r)$  \\ \hline
    &  \multicolumn{4}{c}{By the $P_2$ WG finite element, $k=2$ in \eqref{V-T}  }   \\
\hline  
 1&    0.454E+00 & --- &    0.154E+02 & --- \\
 2&    0.110E+00 &  2.0&    0.114E+02 &  0.4\\
 3&    0.121E-01 &  3.2&    0.755E+01 &  0.6\\
 4&    0.120E-02 &  3.3&    0.396E+01 &  0.9\\
 5&    0.138E-03 &  3.1&    0.200E+01 &  1.0\\
 6&    0.167E-04 &  3.0&    0.100E+01 &  1.0\\
\hline 
    &  \multicolumn{4}{c}{By the $P_3$ WG finite element, $k=3$ in \eqref{V-T}  }   \\
\hline  
 1&    0.339E+00 & --- &    0.234E+02 & --- \\
 2&    0.414E-01 &  3.0&    0.120E+02 &  1.0\\
 3&    0.217E-02 &  4.3&    0.320E+01 &  1.9\\
 4&    0.123E-03 &  4.1&    0.840E+00 &  1.9\\
 5&    0.696E-05 &  4.1&    0.213E+00 &  2.0\\
 \hline  
    &  \multicolumn{4}{c}{By the $P_4$ WG finite element, $k=4$ in \eqref{V-T}  }   \\
\hline  
 1&    0.104E+00 & --- &    0.314E+02 & --- \\
 2&    0.566E-02 &  4.2&    0.541E+01 &  2.5\\
 3&    0.275E-03 &  4.4&    0.106E+01 &  2.3\\
 4&    0.912E-05 &  4.9&    0.148E+00 &  2.8\\
 5&    0.277E-06 &  5.0&    0.189E-01 &  3.0\\
\hline 
    \end{tabular}%
\end{table}

\begin{table}[H]
  \centering  \renewcommand{\arraystretch}{1.02}
  \caption{Error profile for \eqref{sol2}, on triangular grids in Figure \ref{f-g2}. }
  \label{t4}
\begin{tabular}{c|cc|cc}
\hline
  $G_i$ & \quad $\| u-u_h\|_{0}$ & $O(h^r)$ & \  $\3bar u-u_h\3bar $& $O(h^r)$  \\ \hline
    &  \multicolumn{4}{c}{By the $P_2$ WG finite element, $k=2$ in \eqref{V-T}  }   \\
\hline  
 1&    0.461E+00 & --- &    0.266E+02 & --- \\
 2&    0.130E+00 &  1.8&    0.265E+02 &  0.0\\
 3&    0.187E-01 &  2.8&    0.185E+02 &  0.5\\
 4&    0.179E-02 &  3.4&    0.985E+01 &  0.9\\
 5&    0.173E-03 &  3.4&    0.500E+01 &  1.0\\
 6&    0.199E-04 &  3.1&    0.251E+01 &  1.0\\
\hline 
    &  \multicolumn{4}{c}{By the $P_3$ WG finite element, $k=3$ in \eqref{V-T}  }   \\
\hline  
 1&    0.309E+00 & --- &    0.505E+02 & --- \\
 2&    0.411E-01 &  2.9&    0.291E+02 &  0.8\\
 3&    0.212E-02 &  4.3&    0.785E+01 &  1.9\\
 4&    0.120E-03 &  4.1&    0.208E+01 &  1.9\\
 5&    0.690E-05 &  4.1&    0.529E+00 &  2.0\\
 \hline  
    &  \multicolumn{4}{c}{By the $P_4$ WG finite element, $k=4$ in \eqref{V-T}  }   \\
\hline  
 1&    0.103E+00 & --- &    0.697E+02 & --- \\
 2&    0.554E-02 &  4.2&    0.128E+02 &  2.4\\
 3&    0.299E-03 &  4.2&    0.257E+01 &  2.3\\
 4&    0.116E-04 &  4.7&    0.359E+00 &  2.8\\
 5&    0.369E-06 &  5.0&    0.461E-01 &  3.0\\
\hline 
    \end{tabular}%
\end{table}%

\begin{figure}[H]
 \begin{center}\setlength\unitlength{1.0pt}
\begin{picture}(340,110)(0,0) 
 \put(0,102){$G_1:$}  \put(115,102){$G_2:$} \put(230,102){$G_3:$} 
  
  \put(0,-10){\includegraphics[width=330pt]{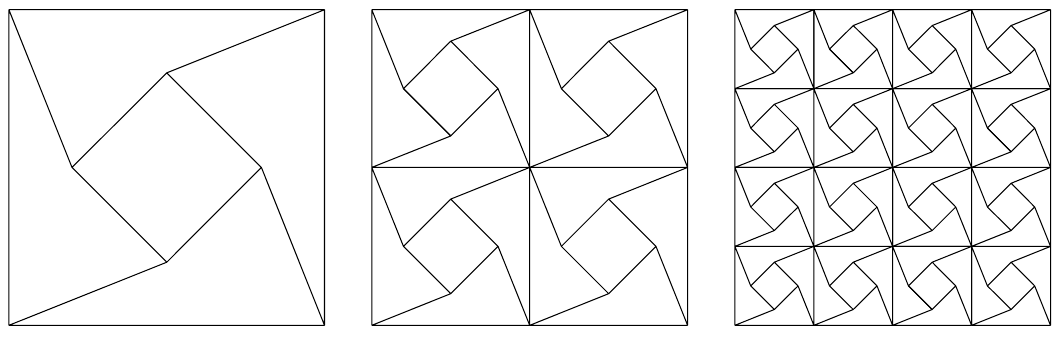}}   
 \end{picture}\end{center}
\caption{The quadrilateral grids, used in Tables \ref{t5}--\ref{t6}. }\label{f-g3}
\end{figure}
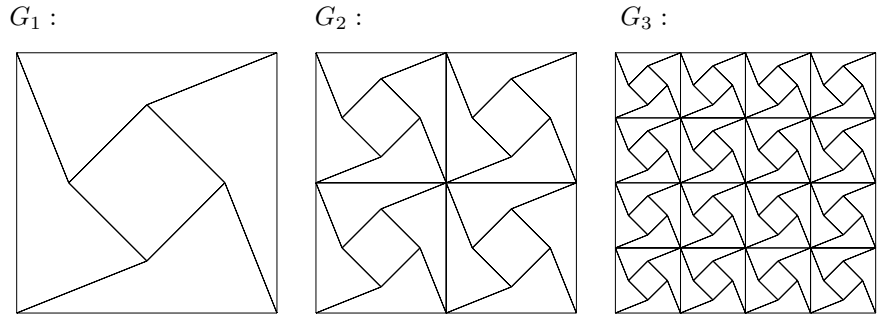

\begin{table}[H]
  \centering  \renewcommand{\arraystretch}{1.1}
  \caption{Error profile for computing \eqref{sol1}, on quadrilateral grids in Figure \ref{f-g3}. }
  \label{t5}
\begin{tabular}{c|cc|cc}
\hline
  $G_i$ & \quad $\| u-u_h\|_{0}$ & $O(h^r)$ & \  $\3bar u-u_h\3bar $& $O(h^r)$  \\ \hline
    &  \multicolumn{4}{c}{By the $P_2$ WG finite element, $k=2$ in \eqref{V-T}  }   \\
\hline  
 1&    0.556E+00 & --- &    0.102E+03 & --- \\
 2&    0.105E+00 &  2.4&    0.623E+02 &  0.7\\
 3&    0.133E-01 &  3.0&    0.312E+02 &  1.0\\
 4&    0.161E-02 &  3.0&    0.160E+02 &  1.0\\
 5&    0.199E-03 &  3.0&    0.807E+01 &  1.0\\
\hline 
    &  \multicolumn{4}{c}{By the $P_3$ WG finite element, $k=3$ in \eqref{V-T}  }   \\
\hline  
 1&    0.239E+00 & --- &    0.160E+03 & --- \\
 2&    0.281E-01 &  3.1&    0.608E+02 &  1.4\\
 3&    0.244E-02 &  3.5&    0.206E+02 &  1.6\\
 4&    0.156E-03 &  4.0&    0.533E+01 &  1.9\\
 5&    0.980E-05 &  4.0&    0.134E+01 &  2.0\\
 \hline  
    &  \multicolumn{4}{c}{By the $P_4$ WG finite element, $k=4$ in \eqref{V-T}  }   \\
\hline  
 1&    0.309E+00 & --- &    0.461E+03 & --- \\
 2&    0.151E-01 &  4.4&    0.801E+02 &  2.5\\
 3&    0.502E-03 &  4.9&    0.108E+02 &  2.9\\
 4&    0.159E-04 &  5.0&    0.131E+01 &  3.0\\
\hline 
    \end{tabular}%
\end{table}%

\begin{table}[H]
  \centering  \renewcommand{\arraystretch}{1.1}
  \caption{Error profile for computing \eqref{sol2}, on quadrilateral grids in Figure \ref{f-g3}. }
  \label{t6}
\begin{tabular}{c|cc|cc}
\hline
  $G_i$ & \quad $\| u-u_h\|_{0}$ & $O(h^r)$ & \  $\3bar u-u_h\3bar $& $O(h^r)$  \\ \hline
    &  \multicolumn{4}{c}{By the $P_2$ WG finite element, $k=2$ in \eqref{V-T}  }   \\
\hline  
 1&    0.581E+00 & --- &    0.211E+03 & --- \\
 2&    0.104E+00 &  2.5&    0.131E+03 &  0.7\\
 3&    0.135E-01 &  2.9&    0.655E+02 &  1.0\\
 4&    0.164E-02 &  3.0&    0.336E+02 &  1.0\\
 5&    0.200E-03 &  3.0&    0.169E+02 &  1.0\\
\hline 
    &  \multicolumn{4}{c}{By the $P_3$ WG finite element, $k=3$ in \eqref{V-T}  }   \\
\hline  
 1&    0.231E+00 & --- &    0.335E+03 & --- \\
 2&    0.301E-01 &  2.9&    0.127E+03 &  1.4\\
 3&    0.243E-02 &  3.6&    0.430E+02 &  1.6\\
 4&    0.156E-03 &  4.0&    0.111E+02 &  1.9\\
 5&    0.979E-05 &  4.0&    0.280E+01 &  2.0\\
 \hline  
    &  \multicolumn{4}{c}{By the $P_4$ WG finite element, $k=4$ in \eqref{V-T}  }   \\
\hline  
 1&    0.316E+00 & --- &    0.956E+03 & --- \\
 2&    0.152E-01 &  4.4&    0.166E+03 &  2.5\\
 3&    0.503E-03 &  4.9&    0.223E+02 &  2.9\\
 4&    0.182E-04 &  4.8&    0.271E+01 &  3.0\\
 5&    0.369E-06 &  5.0&    0.461E-01 &  3.0\\
\hline 
    \end{tabular}%
\end{table}%

\begin{figure}[H]
 \begin{center}\setlength\unitlength{1.0pt}
\begin{picture}(340,110)(0,0) 
 \put(0,102){$G_1:$}  \put(115,102){$G_2:$} \put(230,102){$G_3:$} 
  
  \put(0,-10){\includegraphics[width=330pt]{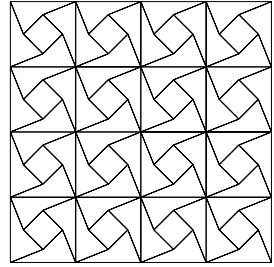}}   
 \end{picture}\end{center}
\caption{The pentagonal grids, used in Tables \ref{t7}--\ref{t8}. }\label{f-g4}
\end{figure}

\begin{table}[H]
  \centering  \renewcommand{\arraystretch}{1.1}
  \caption{Error profile for computing \eqref{sol1}, on pentagonal grids in Figure \ref{f-g4}. }
  \label{t7}
\begin{tabular}{c|cc|cc}
\hline
  $G_i$ & \quad $\| u-u_h\|_{0}$ & $O(h^r)$ & \  $\3bar u-u_h\3bar $& $O(h^r)$  \\ \hline
    &  \multicolumn{4}{c}{By the $P_2$ WG finite element, $k=2$ in \eqref{V-T}  }   \\
\hline  
 1&    0.746E+00 & --- &    0.525E+02 & --- \\
 2&    0.202E+00 &  1.9&    0.291E+02 &  0.9\\
 3&    0.274E-01 &  2.9&    0.150E+02 &  1.0\\
 4&    0.346E-02 &  3.0&    0.756E+01 &  1.0\\
 5&    0.429E-03 &  3.0&    0.380E+01 &  1.0\\
 6&    0.536E-04 &  3.0&    0.190E+01 &  1.0\\
\hline 
    &  \multicolumn{4}{c}{By the $P_3$ WG finite element, $k=3$ in \eqref{V-T}  }   \\
\hline  
 1&    0.940E+00 & --- &    0.136E+03 & --- \\
 2&    0.141E+00 &  2.7&    0.708E+02 &  0.9\\
 3&    0.919E-02 &  3.9&    0.187E+02 &  1.9\\
 4&    0.590E-03 &  4.0&    0.486E+01 &  1.9\\
 5&    0.369E-04 &  4.0&    0.123E+01 &  2.0\\ 
\hline 
    \end{tabular}%
\end{table}%

\begin{table}[H]
  \centering  \renewcommand{\arraystretch}{1.1}
  \caption{Error profile for computing \eqref{sol2}, on pentagonal grids in Figure \ref{f-g4}. }
  \label{t8}
\begin{tabular}{c|cc|cc}
\hline
  $G_i$ & \quad $\| u-u_h\|_{0}$ & $O(h^r)$ & \  $\3bar u-u_h\3bar $& $O(h^r)$  \\ \hline
    &  \multicolumn{4}{c}{By the $P_2$ WG finite element, $k=2$ in \eqref{V-T}  }   \\
\hline  
 1&    0.674E+01 & --- &    0.138E+03 & --- \\
 2&    0.230E+00 &  4.9&    0.571E+02 &  1.3\\
 3&    0.270E-01 &  3.1&    0.287E+02 &  1.0\\
 4&    0.336E-02 &  3.0&    0.143E+02 &  1.0\\
 5&    0.427E-03 &  3.0&    0.716E+01 &  1.0\\
 6&    0.535E-04 &  3.0&    0.358E+01 &  1.0\\
\hline 
    &  \multicolumn{4}{c}{By the $P_3$ WG finite element, $k=3$ in \eqref{V-T}  }   \\
\hline  
 1&    0.978E+00 & --- &    0.288E+03 & --- \\
 2&    0.145E+00 &  2.8&    0.143E+03 &  1.0\\
 3&    0.918E-02 &  4.0&    0.372E+02 &  1.9\\
 4&    0.593E-03 &  4.0&    0.965E+01 &  1.9\\
 5&    0.372E-04 &  4.0&    0.244E+01 &  2.0\\
 \hline   
    \end{tabular}%
\end{table}%

\end{document}